\documentclass[a4paper]{article}
\usepackage{amsmath, amssymb, graphicx, parskip, amsthm, fancyhdr, dsfont, algorithm, algpseudocode, hyperref, bm, tikz}
\usetikzlibrary{positioning}

\newtheorem{thm}{Theorem}

\theoremstyle{definition}

\newtheorem{rmk}{Remark}

\newcommand{\bbP}{\mathbb{P}}
\newcommand{\bbE}{\mathbb{E}}
\newcommand{\bbN}{\mathbb{N}}
\newcommand{\bbR}{\mathbb{R}}

\title{Bernoulli factories and duality in Wright--Fisher and Allen--Cahn models of population genetics}
\author{Jere Koskela${}^{\text{a}, \text{b},*}$, Krzysztof {\L}atuszy\'nski${}^{\text{b}}$, and Dario Span\`o${}^{\text{b}}$\\
	\small jere.koskela@newcastle.ac.uk, k.g.latuszynski@warwick.ac.uk, d.spano@warwick.ac.uk \\
	\small ${}^{\text{a}}$ School of Mathematics, Statistics and Physics, Newcastle University, Newcastle upon Tyne, NE1 7RU, UK \\
	\small ${}^{\text{b}}$ Department of Statistics, University of Warwick, Coventry CV4 7AL, UK \\
	\small *Corresponding author
}
\date{\today}

\begin{document}

\maketitle

\begin{abstract}
Mathematical models of genetic evolution often come in pairs, connected by a so-called duality relation.
The most seminal example are the Wright--Fisher diffusion and the Kingman coalescent, where the former describes the stochastic evolution of neutral allele frequencies in a large population forwards in time, and the latter describes the genetic ancestry of randomly sampled individuals from the population backwards in time.
As well as providing a richer description than either model in isolation, duality often yields equations satisfied by quantities of interest.
We employ the so-called Bernoulli factory---a celebrated tool in simulation-based computing---to derive duality relations for broad classes of genetics models.
As concrete examples, we present Wright--Fisher diffusions with general drift functions, and Allen--Cahn equations with general, nonlinear forcing terms.
The drift and forcing functions can be interpreted as the action of frequency-dependent selection.
To our knowledge, this work is the first time a connection has been drawn between Bernoulli factories and duality in models of population genetics.
\end{abstract}



\section{Introduction}

The Bernoulli factory problem is to construct a realisation of a Bernoulli$(f(p))$ random variable (or an $f(p)$-coin) using an almost surely finite number of independent $p$-coins, where $f : [0, 1] \mapsto [0, 1]$ is a known function but $p \in [0, 1]$ is unknown.
The special case $f(p) = 1/2$ was formulated and solved by John von Neumann in 1951 \cite{vonNeumann:1951}.
Later, Keane and O'Brien provided a necessary and sufficient condition for a given function $f$ to have a Bernoulli factory \cite{keane/obrien:1994}.
In brief, $f$ has a Bernoulli factory if and only if it is identically equal to zero or one, or it is continuous and \emph{polynomially bounded}, i.e.
\begin{equation}\label{polynomially_bounded}
\min\{ f( p ), 1 - f( p ) \} \geq \min\{ p, 1 - p \}^n
\end{equation}
for all $p \in [ 0, 1 ]$ and some $n \in \bbN$.
However, the proof of Keane and O'Brien is only partly constructive: it relies on a recursively defined sequence whose explicit solution is intractable.
Constructions of algorithms have relied on approximation of $f$ by Bernstein polynomials, which are naturally associated with $p$-coins \cite{holtzetal:2011, latuszynskietal:2011, nacu/peres:2005}, or on other series expansions of $f$ with non-negative coefficients \cite{mendo:2019}.

Many seminal models of population genetics rely on the random propagation of alleles from one generation to the next.
A prototypical example is the Wright--Fisher model, in which a population of fixed size $N \in \bbN$ evolves in discrete generations.
Individuals carry one of two alleles, $a$ or $A$, and each individual inherits the allele of a parent which it samples independently and uniformly from the previous generation.
The mechanism of sampling alleles by sampling parents ensures that the model carries information of the forward-in-time evolution of allele frequencies, as well as the backward-in-time genealogies of samples of individuals.
Frequently, these two modelling perspectives satisfy a duality relation which renders both models more tractable than they would be in isolation.
We direct interested readers to e.g.\ \cite{Durrett08} for an introduction to Wright--Fisher and genealogical models in population genetics.

In the absence of mutation, inheriting an allele from a uniformly sampled parent models \emph{neutral} evolution, where the conditional mean allele frequency in a generation equals that in the previous generation.
Non-neutral models in which the mean is not constant can be obtained by sampling offspring alleles as $f(p)$-coins when the allele frequency in the previous generation is $p$, and $f$ models so-called \emph{frequency-dependent selection}.
In order to retain the aforementioned backward-in-time genealogical picture and its associated duality, it is desirable to generate the $f(p)$-coins in a given generation by sampling parental alleles from the previous generation.
Since the allele frequency in the parental generation is $p$, parental alleles can be thought of as $p$-coins, motivating a connection to Bernoulli factories.
Our contribution is to use a Bernoulli factory to extend two standard models of population genetics to more general settings than has been done previously.
They are (i) the non-neutral Wright--Fisher diffusion and its ancestral selection graph dual \cite{corderoetal:2022, Etheridge10, casanova/spano:2018, Krone97}, for which the function $f$ models frequency-dependent selection as described above, and (ii) the Allen--Cahn PDE which models stationary allele frequencies in a spatial continuum, in which $f$ appears as external forcing \cite{etheridgeetal:2022}.

The remainder of the manuscript is organised as follows.
In Section \ref{factory}, we review the Bernoulli factory of \cite{keane/obrien:1994} which turns out to be convenient for our purposes.
In Sections \ref{wright_fisher} and \ref{allen_cahn} we introduce the Wright--Fisher diffusion with frequency--dependent selection and the Allen--Cahn equation.
In each case, we also demonstrate how Bernoulli factories facilitate the construction of very large classes of these models.
Section \ref{discussion} concludes with a discussion on connections to earlier results, as well as some potential extensions.

\section{The Keane--O'Brien factory}\label{factory}

Let $f : [0,1] \mapsto [0,1]$ be continuous and polynomially bounded as in \eqref{polynomially_bounded}.
Let $\bar{X}_n( p )$ be the sample mean of $n$ independent $p$-coins.
Define sequences of functions $f_k$ and integers $\eta( f, k )$ as follows: set $f_1( p ) := f( p )$, and
\begin{equation*}
f_{ k + 1 }( p ) := \frac{ 4 }{ 3 } \Bigg( f_k( p ) - \frac{ 1 }{ 4 } \bbP\Big( f_k( \bar{X}_{ \eta( f, k ) }( p ) ) \geq \frac{ 1 }{ 2 } \Big) \Bigg),
\end{equation*}
where each $\eta( f, k )$ is finite, independent of $p,$ and large enough that
\begin{equation*}
  f_k( p ) - \frac{ 1 }{ 4 } \bbP\Big( f_k( \bar{X}_{ \eta( f, k ) }( p ) ) \geq \frac{ 1 }{ 2 } \Big)  \in [0, 3/4].
\end{equation*}
By \cite[page 218]{keane/obrien:1994} such an $\eta( f, k )$ exists, and moreover, each $f_k$ is polynomially bounded.

Let $L \sim \operatorname{Geo}(1/4)$ be independent of all $p$-coins.
Then 
\begin{equation*}
\mathds{1}\Big\{ f_L( \bar{ X }_{ \eta( f, L ) }( p ) ) \geq \frac{ 1 }{ 2 } \Big\} \sim \operatorname{Ber}( f( p ) )
\end{equation*}
is a Bernoulli factory for $f$ \cite[pages 217--219]{keane/obrien:1994} based on the series expansion
\begin{align}
f( p ) &= \sum_{ k = 1 }^{ \infty } \Big( \frac{ 3 }{ 4 } \Big)^{ k - 1 } \frac{ 1 }{ 4 } \bbP\Big( f_k( \bar{X}_{ \eta( f, k ) }( p ) ) \geq \frac{ 1 }{ 2 } \Big) \notag \\
&= \sum_{ k = 1 }^{ \infty } \Big( \frac{ 3 }{ 4 } \Big)^{ k - 1 } \frac{ 1 }{ 4 } \sum_{ j = 0 }^{ \eta( f, k ) } \binom{ \eta( f, k ) }{ j } p^j ( 1 - p )^{ \eta( f, k ) - j } \mathds{ 1 }\Big\{ f_k \Big( \frac{ j }{ \eta( f, k ) } \Big) \geq \frac{ 1 }{ 2 } \Big\}, \label{pls_expansion}
\end{align}
which converges uniformly in $p$. 
Pseudocode for this Bernoulli factory is shown in Algorithm~\ref{alg_1}.
\begin{algorithm}
\caption{Bernoulli factory for continuous $f : [0, 1] \mapsto [0, 1]$ satisfying \eqref{polynomially_bounded}.}
\label{alg_1}
\begin{algorithmic}[1]
\Require Sequences $\{ f_k \}_{k \geq 1}$, $\{ \eta( f, k ) \}_{k \geq 1}$.
\State Sample $L \sim \operatorname{Geo}(1 / 4)$. 
\State Sample $X_{\eta( f, L ) }( p ) \sim \operatorname{Bin}( \eta( f, L ), p )$.
\If{$f_L( X_{\eta( f, L ) } ( p ) / L ) \geq 1/2$} 
	\State Return 1.
\Else
	\State Return 0.
\EndIf
\end{algorithmic}
\end{algorithm}
We will refer to Bernoulli factories based on \eqref{pls_expansion} as Keane--O'Brien factories.
They have two features which which will turn out to be essential.
\begin{rmk}\label{rmk_1}
The Keane--O'Brien factory is defined for all $p \in [ 0, 1 ]$ and $f( p ) \in [ 0, 1 ]$.
Factories based on Bernstein polynomial approximation typically have to constrain the range (and sometimes the domain) of $f$ to (subsets of) $(0, 1)$ to avoid degenerate distributions \cite{holtzetal:2011, latuszynskietal:2011, nacu/peres:2005}.
In population genetics applications, it is essential to allow the whole range $p \in [0, 1]$, and desirable to allow $f( 0 ) = 0$ and $f( 1 ) = 1$ to model fixation.
\end{rmk}
\begin{rmk}\label{rmk_2}
The random variable $L$, and hence the number of $p$-coins $\eta( f, L )$ needed to determine the allele of a child, is independent of the realisations of those coins.
This will facilitate the construction of ancestral graphs containing all possible ancestors before the alleles carried by any of those ancestors are known.
\end{rmk}
Any Bernoulli factory for which Remark \ref{rmk_2} holds can be thought of as a mixture of finite factories, that is, a mixture of Bernoulli factories with a deterministic number of coins each.

\section{The frequency-dependent Wright--Fisher diffusion} \label{wright_fisher}

Consider a population of $N$ individuals $\{ \mathbf{Z}_k^{ ( N ) } \}_{ k \geq 0 }$ evolving in discrete time, where the state of the system at time $k$ is $\mathbf{ Z }_k^{ ( N ) } := ( Z_1^{ ( N ) }( k ), \ldots, Z_N^{ ( N ) }( k ) )$ with $Z_k^{ ( N ) } \in \{ a, A \}$.
Let 
\begin{equation*}
\tilde{ Y }_k^{ ( N ) } := \frac{ 1 }{ N } \sum_{ i = 1 }^N \mathds{ 1 }_{ \{ A \} }( Z_i^{ ( N ) }( k ) )
\end{equation*}
be the frequency of the $A$ allele in generation $k$.
Each individual in generation $k + 1$ samples its allele conditionally independently given $\tilde{ Y }_k^{ ( N ) }$, where the probability of an $A$ allele is $f^{ ( N ) }( Y_k^{ ( N ) } )$ with
\begin{equation*}
f^{ ( N ) }( p ) := \Big( 1 - \frac{ \sigma }{ N } \Big) p + \frac{ \sigma }{ N } f( p )
\end{equation*}
for a fixed constant $\sigma > 0$, a continuous function $f : [ 0, 1 ] \mapsto [ 0, 1 ]$ satisfying \eqref{polynomially_bounded}, and where we assume the population size is large enough that $N \geq \sigma$.

Concretely, the realisation of each allele is implemented as follows.
With probability $1 - \sigma / N$, a given individual samples one parent and inherits its allele.
With the complementary probability, the individual samples an independent copy of $L$, followed by $\eta( f, L )$ parents chosen uniformly at random with replacement.
Conditional on $L$ and $\tilde{ Y }_k^{ ( N ) }$, the allele of the individual is $A$ if event $\{ f_L( \bar{ X }_{ \eta( f, L ) }( \tilde{ Y }_k^{ ( N ) } ) ) \geq 1 / 2 \}$ happens.
By \eqref{pls_expansion}, the marginal probability of an $A$ allele is $f^{ ( N ) }( \tilde{ Y }_k^{ ( N ) } )$.
A realisation of the process of children sampling parents is depicted in Figure \ref{wright_fisher_fig}.

\begin{figure}[!h]
	\centering
	\begin{tikzpicture}
		\node[circle,draw] (X10) {};
		\node[circle,draw] (X20) [below=of X10] {};
		\node[circle,draw] (X30) [below=of X20] {};
		\node[circle,draw] (X40) [below=of X30] {};
		\node[circle,draw] (X11) [right=of X10] {};
		\node[circle,draw] (X21) [below=of X11] {};
		\node[circle,draw] (X31) [below=of X21] {};
		\node[circle,draw] (X41) [below=of X31] {};
		\node[circle,draw] (X12) [right=of X11] {};
		\node[circle,draw] (X22) [below=of X12] {};
		\node[circle,draw] (X32) [below=of X22] {};
		\node[circle,draw] (X42) [below=of X32] {};
		\node[circle,draw] (X13) [right=of X12] {};
		\node[circle,draw] (X23) [below=of X13] {};
		\node[circle,draw] (X33) [below=of X23] {};
		\node[circle,draw] (X43) [below=of X33] {};
		\node[circle,draw] (X14) [right=of X13] {};
		\node[circle,draw] (X24) [below=of X14] {};
		\node[circle,draw] (X34) [below=of X24] {};
		\node[circle,draw] (X44) [below=of X34] {};
		\path (X20) edge[-latex] (X10);
		\path (X21) edge[-latex] (X10);
		\path (X22) edge[-latex] (X12);
		\path (X23) edge[-latex] (X13);
		\path (X24) edge[-latex] (X13);
		\path (X30) edge[-latex] (X20);
		\path (X31) edge[-latex] (X21);
		\path (X32) edge[-latex] (X22);
		\path (X33) edge[-latex] (X22);
		\path (X33) edge[-latex, in=-70, out=70] (X23);
		\path (X33) edge[-latex, in=-110, out = 110] (X23);
		\path (X34) edge[-latex] (X24);
		\path (X40) edge[-latex] (X31);
		\path (X41) edge[-latex] (X31);
		\path (X42) edge[-latex] (X31);
		\path (X43) edge[-latex] (X33);
		\path (X44) edge[-latex] (X34);	
	\end{tikzpicture}
\caption{Illustration of $\{ \mathbf{Z}_k^{ ( N ) } \}_{ k \geq 0 }$ with $N = 5$ and 4 generations.
The second generation from the bottom includes a merger of three lineages, as well as an individual which samples three parents, two of which ended up overlapping.
The generation at the top includes two binary mergers.
Alleles for all individuals in all generations could be generated given alleles for the generation at the top, as well as a rule for deciding the allele of the three-parent child.}
\label{wright_fisher_fig}
\end{figure}
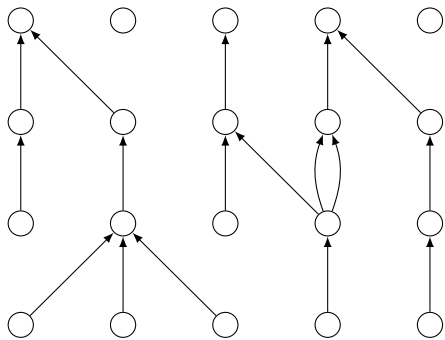

Sampling parents with replacement is essential for the connection with Bernoulli factories, as each parent can then be thought of as an i.i.d.\ coin on the two-point space of alleles $\{a, A\}$.
However, sampling with replacement also permits a child to choose the same individual as a parent many times, which is unnatural from a biological point of view.
Formally, the $\eta(f, L)$ parents are merely a mechanism for generating alleles for children in a parent-driven way which turns out to facilitate duality with a genealogical process.
A biologically interpretable genealogy with one parent per child can be recovered by selecting, for each multi-parent child, one of the child-parent connections to denote the ``real" parent, with all other child-parent connections designated ``virtual" and hence non-biological.
The way in which the real parent is chosen can be arbitrary as long as it has no impact on the distribution of the child allele, and is independent of the child-parent connections of other individuals.
However, retention of virtual parents is essential for constructing both the genealogical and allele frequency processes simultaneously, and hence obtaining the duality result in Theorem \ref{duality_theorem}.

Our aim is to verify convergence of the allele frequency process $\{ \tilde{ Y }_k^{ ( N ) } \}_{ k \geq 0 }$ in the infinite population limit.
To that end, we define the continuous-time process $( Y_t^{ ( N ) } )_{ t \geq 0 }$ via
\begin{equation*}
Y_t^{ ( N ) } := \tilde{ Y }_{ \lfloor N t \rfloor }^{ ( N ) }.
\end{equation*}
\begin{thm}\label{wf_theorem}
Let $f$ be polynomially bounded and Lipschitz continuous on $[0, 1]$.
Then, in the Skorokhod topology, $( Y_t^{ ( N ) } )_{ t \geq 0 } \to ( Y_t )_{ t \geq 0 }$ weakly as $N \to \infty$, where $Y_t$ solves
\begin{equation}\label{wf_diffusion}
\mathrm{d}Y_t = \sigma ( f( Y_t ) - Y_t ) \mathrm{d}t + \sqrt{ Y_t ( 1 - Y_t ) } \mathrm{d}W_t
\end{equation}
subject to appropriate initial conditions, where $( W_t )_{ t \geq 0 }$ is a Brownian motion.
\end{thm}
\begin{proof}
Existence, uniqueness, and the Feller property of the putative limiting process all follow from \cite[Theorem 2.8]{ethier/kurtz:1986} because the drift function $\sigma ( f( y ) - y )$ is Lipschitz continuous.

The discrete-time process $\bar{Y}_k^{ ( N ) }$ has discrete generator
\begin{align*}
L^N h( y ) := \sum_{ k = 0 }^N &\binom{ N }{ k } \Big( \frac{ \sigma }{ N } \Big)^k \Big( 1 - \frac{ \sigma }{ N } \Big )^{ N - k } \sum_{ x = 0 }^{ N - k } \binom{ N - k }{ x } y^x ( 1 - y )^{ N - k - x }  \\
&\times \sum_{ z = 0 }^k \binom{ k }{ z } f( y )^z ( 1 - f( y ) )^{ k - z } h\Big( \frac{ x + z }{ N } \Big) - h( y ),
\end{align*}
which is well-defined for any bounded, measurable function $h : [ 0, 1 ] \to \bbR$.
To show the claimed convergence, let $h$ be a $C^2( [ 0, 1 ] )$ function with bounded third derivatives, which is a convergence-determining class.
Expanding $h$ around $y$ on the right-hand side yields
\begin{align*}
L^N h( y ) = {}&\sum_{ k = 0 }^N \binom{ N }{ k } \Big( \frac{ \sigma }{ N } \Big)^k \Big( 1 - \frac{ \sigma }{ N } \Big )^{ N - k } \sum_{ x = 0 }^{ N - k } \binom{ N - k }{ x } y^x ( 1 - y )^{ N - k - x } \\
&\times \sum_{ z = 0 }^k \binom{ k }{ z } f( y )^z ( 1 - f( y ) )^{ k - z } \Bigg[ \Big( \frac{ x + z }{ N } - y \Big) h'( y ) + \frac{ 1 }{ 2 } \Big( \frac{ x + z }{ N } - y \Big)^2 h''( y ) \Bigg],
\end{align*}
up to terms which are of lower order since $h'''$ is bounded.
Evaluating the binomial expectations yields
\begin{equation*}
N L^N h( y ) = \sigma ( f( y ) - y ) h'( y ) + \frac{ 1 }{ 2 } y ( 1 - y ) h''( y ) + o( 1 ),
\end{equation*}
as $N \to \infty$.
The proof of weak convergence is completed by using the pseudo-Poisson approximation of \cite[Theorem 19.28]{kallenberg:2002} with time step $1 / N$, because the limiting diffusion is Feller.
\end{proof}

\begin{rmk}
The requirement of a Lipschitz drift could be slightly relaxed by using the more cumbersome conditions for a Feller semigroup given in \cite[Theorem 3.3]{xietal:2019}.
They cover drifts satisfying a condition akin to
\begin{equation*}
| f( y ) - f( z ) | \leq - C | y - z | \log( | y - z | ),
\end{equation*}
or small variations thereof, where $C > 0$ is a constant.
See \cite[equations (14) and (15), as well as Remark 2.3]{xietal:2019} for a precise class of non-Lipschitz functions which can be handled.
\end{rmk}

Next we consider a sample of $n \in \bbN$ individuals from a given generation (which we arbitrarily set as generation 0) in the pre-limiting Wright--Fisher model $\{ \mathbf{Z}_k^{ ( N ) } \}_{ k \geq 0 }$.
We define the \emph{ancestral process} $A_k^{ ( N ) }$ as the number of lineages which are ancestral to the sample $k$ generations in the past.
The number of lineages decreases by one when two lineages find a common ancestor, decreases by more than one if multiple lineages find common ancestors (which will happen with negligible probability when $N \to \infty$), or increase by $\eta( f, L ) - 1$ whenever a lineage samples $\eta( f, L )$ ancestors.
In the pre-limiting particle system, any number of these events can co-occur in one generation, particularly when $n \geq 3$.
But transitions other than isolated binary mergers and single multifurcations turn out to vanish in a suitably rescaled infinite population limit.
To that end, we define the continuous time Markov jump process
\begin{equation}\label{ancestral_limit}
A_t := \lim_{ N \to \infty } A_{ \lfloor N t \rfloor }^{ ( N ) }
\end{equation}
whose existence we prove next.
\begin{thm}\label{ancestral_thm}
The limit in \eqref{ancestral_limit} exists, and $( A_t )_{ t \geq 0 }$ has generator
\begin{equation*}
G h( n ) = \binom{ n }{ 2 } [ h( n - 1 ) - h( n ) ] + \sigma n \sum_{ k = 1 }^{ \infty } \Big( \frac{ 3 }{ 4 } \Big)^{ k - 1 } \frac{ 1 }{ 4 } [ h( n + \eta( f, k ) - 1 ) - h( n ) ].
\end{equation*}
\end{thm}
\begin{proof}
The pre-limiting, discrete-time process $\{ A_{ k }^{ ( N ) } \}_{ k \geq 0 }$ undergoes a myriad of transitions involving subsets of individuals finding common ancestors, as well as individuals branching into many potential ancestors, potentially in the same generation.
However, only transitions with a per-generation probability $\Theta(1 / N)$ will contribute to the time-rescaled limit with a finite rate.
Events with probability $o(1/N)$ will not occur in the limit at all, while events with probability $\omega( 1 / N )$ will appear at a dense set of times.
However, it will turn out that the latter only result in identity transitions in $A_t$, and hence do not affect the limit.

The probability of two lineages originating from a common ancestor one generation earlier, with neither being involved in a branching event, is
\begin{equation}\label{merger_rate}
\Big( 1 - \frac{ \sigma }{ N } \Big)^2 \frac{ 1 }{ N } = \frac{ 1 }{ N } + o( 1 / N ),
\end{equation}
implying that two lineages will merge to a common ancestor at rate 1 in the limit.
A triple merger, or more than one simultaneous merger, has probability at most
\begin{equation*}
\Big( 1 - \frac{ \sigma }{ N } \Big)^3 \frac{ 1 }{ N^2 } = o( 1 / N ),
\end{equation*}
and hence will not appear in the limit.

A single individual branches into $\eta( f, k )$ ancestral lineages with probability
\begin{equation}\label{branching_rate}
\frac{ \sigma }{ N } \Big( \frac{ 3 }{ 4 } \Big)^{k - 1} \frac{ 1 }{ 4 },
\end{equation}
while any event involving more than two lineages branching in one generation has probability at most
\begin{equation*}
\Big( \frac{ \sigma }{ N } \Big)^2 = o( 1 / N ).
\end{equation*}
Hence, only isolated branching events appear in the limit.
It is also clear from \eqref{merger_rate} and \eqref{branching_rate} that the probability of at least one merger and branching event in one generation is $o( 1 / N )$.

All other transitions involve no mergers or branching events, and hence do not affect the limiting ancestral process.
Noting that there are $\binom{n}{2}$ pairs of individuals to merge with probability \eqref{merger_rate} and $n$ individuals to branch with probabilities \eqref{branching_rate} yields the claimed generator $G$.
\end{proof}

Duality between the Wright--Fisher diffusion \eqref{wf_diffusion} and the ancestral process \eqref{ancestral_limit} is a relation
\begin{equation*}
\bbE_y[ h( Y_t, n ) | Y_0 = y ] = \bbE_n[ h( y, A_t ) | A_0 = n ],
\end{equation*}
where $y \in [ 0, 1 ]$ and $n \in \bbN$ are respective initial conditions, and $h$ is a \emph{duality function}, the specification of which will require some exposition.

We follow \cite{corderoetal:2022} and define the random function $P_t( y )$ as the conditional probability that all $n$ leaves at time zero in the ancestral process carry allele $A$, given that the $A$ allele frequency at time $t$ in the past is $y \in [0, 1]$, and given the realisation of the ancestral process $A_{ 0 : t } := ( A_s )_{ s \in [ 0, t ] }$ started from the $n$ lineages.
For example, in the absence of branching events we have $P_t( y ) = y^{ A_t }$, while when $n = 1$, a single branching event into $\eta( f, k )$ ancestors and no mergers in the history $A_{ 0 : t }$ yields
\begin{equation*}
P_t( y ) := \sum_{ j = 0 }^{ \eta( f, k ) } \binom{ \eta( f, k ) }{ j } y^j ( 1 - y )^{ \eta( f, k ) - j } \mathds{ 1 }\Big\{ f_k \Big( \frac{ j }{ \eta( f, k ) } \Big) \geq \frac{ 1 }{ 2 } \Big\},
\end{equation*}
where $f_k$ is as specified in Section \ref{factory}.
Other patterns of merger and branching events will result in a more complicated Bernstein polynomial of degree $A_t$,
\begin{equation*}
P_t( y ) := \sum_{ j = 0 }^{ A_t } V_t( j ) \binom{ A_t }{ j } y^j ( 1 - y )^{ A_t - j },
\end{equation*}
where $V_t( j )$ is the random Bernstein coefficient which equals the probability that the $n$ leaves all carry allele $A$, given that $j$ of $A_t$ roots do.
As in \cite[Definition 2.12 and Proposition 2.13]{corderoetal:2022}, the vector $( \bm{V}_t )_{ t \geq 0 } := ( V_t( 0 ), \ldots, V_t( n ) )_{ t \geq 0 }$ is also a Markov jump process with transitions
\begin{align}
\bm{v} \mapsto \Bigg( \sum_{ j = 0 }^{ i \wedge \eta( f, k ) } \frac{ \binom{ i }{ j } \binom{ n + \eta( f, k ) - 1 - j }{ \eta( f, k ) - j } }{ \binom{ n + \eta( f, k ) - 1 }{ i } } [ &\mathds{1}\{ f_h( j / \eta( f, k ) ) \geq 1/2 \} v_{ i + 1 - j } \notag \\
&+ \mathds{1}\{ f_k( j / \eta( f, k ) ) < 1/2 \} v_{ i - j } ] \Bigg)_{ i = 0 }^{ n + \eta( f, k ) - 1 } \label{branching_mechanism}
\end{align}
at rate $n \sigma \bbP( L = k )$, and 
\begin{equation}\label{coalescing_mechanism}
\bm{v} \mapsto \Big( \frac{ i }{ n - 1 } v_{ i + 1 } + \frac{ n - 1 - i }{ n - 1 } v_i \Big)_{ i = 0 }^{ n - 1 }
\end{equation}
at rate $\binom{ n }{ 2 }$.
Also, let
\begin{equation*}
\bm{B}_n( x ) := \Bigg( \binom{ n }{ i } x^i ( 1 - x )^{ n - i } \Bigg)_{ i = 0 }^n
\end{equation*}
be the vector of order $n$ Bernstein polynomials, and for vectors $\bm{u}$ and $\bm{v}$ of common length $n + 1$, let 
\begin{equation*}
\langle \bm{u}, \bm{v} \rangle := \sum_{ i = 0 }^n u_i v_i
\end{equation*}
be the usual inner product.

\begin{thm}\label{duality_theorem}
The processes $( Y_t )_{ t \geq 0 }$ and $( V_t )_{ t \geq 0 }$ are dual with duality function
\begin{equation*}
H : ( y, \bm{ v } ) \mapsto \langle \bm{B}_{ \text{dim}( \bm{v} ) - 1 }( y ), \bm{v} \rangle,
\end{equation*}
in that
\begin{equation}\label{bernstein_duality}
\bbE[ \langle \bm{ B }_n( Y_t ), \bm{v} \rangle | Y_0 = y ] = \bbE[ \langle \bm{ B }_{ \text{dim}( \bm{ V }_t ) - 1 }( y ), \bm{ V }_t \rangle | \bm{ V }_0 = \bm{ v } ]
\end{equation}
for all $t \geq 0$, each $y \in [ 0, 1 ]$, every $n \in \bbN$, and each $\bm{v} \in \bbR^{ n + 1 }$.
\end{thm}
\begin{proof}
The proof is a small adaptation of that of \cite[Theorem 2.14]{corderoetal:2022} to our more general setting.
Duality between the coalescing mechanism of the ancestral process \eqref{coalescing_mechanism} and the diffusion coefficient of the Wright--Fisher diffusion \eqref{wf_diffusion} is standard, and we omit it to focus on establishing the same relation for the branching mechanism \eqref{branching_mechanism} and the drift term in \eqref{wf_diffusion}.

Following \cite[Section 4.3]{corderoetal:2022}, let $Y_k^y \sim \text{Bin}( k, y )$ and $K_{ k, i }^n \sim \text{Hyp}( n + k - 1, k, i )$ be independent random variables, where $\text{Hyp}( a, b, c )$ denotes the hypergeometric distribution with $b$ draws without replacement from a population of size $a$ containing $c$ successes.
Then, for $\bm{v} \in \bbR^{ n + 1 }$ we have
\begin{align*}
\partial_y H( y, \bm{v} ) &= n \bbE[ v_{ Y_{ n - 1 }^y + 1 } - v_{ Y_{ n - 1 }^y } ], \\
H( y, \bm{v} ) &= \bbE[ v_{ Y_n^y } ] = \bbE[ ( 1 - y ) v_{ Y_{ n - 1 }^y } + y v_{ Y_{ n - 1 }^y + 1 } ], \\
f( y ) &= \bbE[ \mathds{ 1 }\{ f_L( \bar{ X }_{ \eta( f, L ) }( y ) ) \geq 1 / 2 \} ],
\end{align*}
where $L \sim \text{Geo}(1/4)$ and the third equality is due to \eqref{pls_expansion}.
We also have
\begin{equation}\label{binomial_hypergeometric_identity}
( \eta( f, L ) \bar{ X }_{ \eta( f, L ) }( y ), Y_{ n - 1 }^y ) \stackrel{d}{=} ( K_{ \eta( f, L ), Y_{ n + \eta( f, L ) - 1 }^y }^n, Y_{ n + \eta( f, L ) - 1 }^y - K_{ \eta( f, L ), Y_{ n + \eta( f, L ) - 1 }^y }^n ),
\end{equation}
as can be seen by considering the event that the left-hand side takes value $( j, k )$, which necessitates that $K_{ \eta( f, L ), Y_{ n + \eta( f, L ) - 1 }^y }^n = j$ and $Y_{ n + \eta( f, L ) - 1 }^y = k + j$ on the right-hand side.
Using the shorthand $\eta = \eta(f, L)$, the probability of the event on the left can be written as
\begin{align*}
\binom{ \eta }{ j } y^j ( 1 - y )^{ \eta - j } \binom{ n - 1 }{ k } y^k ( 1 - y )^{ n - 1 - k } = \frac{ \binom{ \eta }{ j } \binom{ n + \eta - 1 - \eta }{ k + j - j } }{ \binom{ n + \eta - 1 }{ k + j } } \binom{ n + \eta - 1 }{ k + j } y^{ k + j } ( 1 - y )^{ n + \eta - 1 - k - j },
\end{align*}
which is the claimed probability of the event on the right.
Intuitively, \eqref{binomial_hypergeometric_identity} expresses the fact that flipping $\eta$ coins and $n - 1$ coins in separate groups is statistically identical to flipping $n + \eta - 1$ coins and randomly allocating them into groups of size $\eta$ and $n - 1$ afterwards.
Using \eqref{binomial_hypergeometric_identity}, the drift of the Wright--Fisher diffusion \eqref{wf_diffusion} satisfies
\begin{align*}
&\sigma ( f( y ) - y ) \partial_y H( y, \bm{ v } ) \\
&= \sigma \bbE[ \mathds{ 1 }\{ f_L( \bar{ X }_{ \eta( f, L ) }( y ) ) \geq 1 / 2 \} - y ] n \bbE[ v_{ Y_{ n - 1 }^y + 1 } - v_{ Y_{ n - 1 }^y } ] \\
&= n \sigma \bbE[ \mathds{ 1 }\{ f_L( \bar{ X }_{ \eta( f, L ) }( y ) ) \geq 1 / 2 \} v_{ Y_{ n - 1 }^y + 1 } \\
&\phantom{= n \sigma \bbE[}- \mathds{ 1 }\{ f_L( \bar{ X }_{ \eta( f, L ) }( y ) ) \geq 1 / 2 \} v_{ Y_{ n - 1 }^y } - y  v_{ Y_{ n - 1 }^y + 1 } + y v_{ Y_{ n - 1 }^y } + v_{ Y_{ n - 1 }^y } - v_{ Y_{ n - 1 }^y }  ] \\
&=  n \sigma ( \bbE[ \mathds{ 1 }\{ f_L( \bar{ X }_{ \eta( f, L ) }( y ) ) \geq 1 / 2 \} v_{ Y_{ n - 1 }^y + 1 } + \mathds{ 1 }\{ f_L( \bar{ X }_{ \eta( f, L ) }( y ) ) < 1 / 2 \} v_{ Y_{ n - 1 }^y } ] - H( y, \bm{v} ) ) \\
&= n \sigma \Bigg( \sum_{ k = 1 }^{ \infty } \Big( \frac{ 3 }{ 4 } \Big)^{ k - 1 } \frac{ 1 }{ 4 } \sum_{ j = 0 }^{ n + \eta( f, k ) - 1 } \bbE\Bigg[ \mathds{ 1 }\Bigg\{ f_k\Bigg( \frac{ K_{ \eta( f, k ), j }^n }{ \eta( f, k ) } \Bigg) \geq 1 / 2 \Bigg\} v_{  j - K_{ \eta( f, k ), j }^n + 1 } \\
&\phantom{\ldots \ldots \ldots }+ \mathds{ 1 }\Bigg\{ f_k\Bigg( \frac{ K_{ \eta( f, k ), j }^n }{ \eta( f, k ) } \Bigg) < 1 / 2 \Bigg\} v_{ j - K_{ \eta( f, k ), j }^n } \Bigg] \binom{ n + \eta( f, k ) - 1 }{ j } y^j ( 1 - y )^{ n + \eta( f, k ) - 1  - j }  - H( y, \bm{v} ) \Bigg),
\end{align*}
which is precisely \eqref{branching_mechanism} applied to the $\bm{v}$-argument of $H( y, \bm{v} ) = \langle \bm{B}_{ \text{dim}( \bm{v} ) - 1 }( y ), \bm{v} \rangle$.
\end{proof}

\section{The Allen--Cahn model of spatial genetics} \label{allen_cahn}

The Allen--Cahn equation on a Lipschitz domain $\Omega \subseteq \bbR^d$ is
\begin{equation}\label{allen_cahn_general}
\partial_t u - \Delta u = \lambda f(u), 
\end{equation}
for a given $\lambda > 0$ and $f : [ 0, 1 ] \mapsto [ 0, 1 ]$, and subject to suitable initial and boundary conditions.
In \cite{etheridgeetal:2022}, the authors consider a model of the spatially structured frequency $u( x, t )$ of an allele $A$ governed by
\begin{align}
\partial_t u - \Delta u &= \frac{ 1 }{ \varepsilon^2 } u ( 1 - u ) ( 2 u - 1 + \nu \varepsilon ) \text{ for } x \in \Omega, t > 0, \label{allen_cahn_equation}\\
\partial_n u &= 0 \text{ for } x \in \partial \Omega, t > 0, \notag \\
u(x, 0) &= u_0( x ) \text{ for } x \in \Omega, \notag
\end{align}
where $\varepsilon > 0$, $\nu > 0$, $\partial_n$ denotes the normal derivative at the boundary, and $u_0 : \Omega \to [0, 1]$ (see also \cite{gooding:2018}).
They construct a solution to \eqref{allen_cahn_equation} by using a particle system in which a single particle started at a location $x \in \Omega$ undergoes a Brownian motion with speed 2, branches into three independent copies at rate $( 1 + \varepsilon \nu) / \varepsilon^2$, and particles are reflected from the boundary.
At a given end time $t > 0$, a leaf particle at location $z \in \Omega$ samples one of two alleles, $a$ and $A$, with respective probabilities $( 1 - u_0( z ), u_0( z ) )$.
All particles sample their alleles independently.
Then, particles propagate their alleles rootwards along the realisation of the Brownian tree.
The allele of an internal branch is decided by a majority vote among its three children unless exactly one child carries allele $A$, in which case the parent branch carries $A$ with probability $2 \nu \varepsilon / ( 3 + 3 \nu \varepsilon )$, and $a$ otherwise.
The probability that the root particle $x \in \Omega$ carries allele $A$ under these dynamics solves \eqref{allen_cahn_equation} \cite[Proposition 2.4]{etheridgeetal:2022}.

The ingredients of the particle system can be read off from \eqref{allen_cahn_equation}.
The Laplacian $\Delta$ is the generator of Brownian motion run at speed 2, the branching rate $( 1 + \varepsilon \nu) / \varepsilon^2$ is an upper bound for the right-hand side, and 
\begin{equation*}
\frac{ 1 }{ \varepsilon^2 } u ( 1 - u ) ( 2 u - 1 + \nu \varepsilon ) = \frac{ 1 + \nu \varepsilon }{ \varepsilon^2 } \Bigg( u^3 + 3 u^2( 1 - u ) + \frac{ 2 \nu \varepsilon }{ 3 + 3 \nu \varepsilon } 3 u ( 1 - u )^2 - u  \Bigg),
\end{equation*}
which demonstrates that the nonlinearity in \eqref{allen_cahn_equation} has an interpretation as the voting system described above.
It would be straightforward to adapt the proof of \cite[Proposition 2.4]{etheridgeetal:2022} to any other polynomial right-hand side of the form
\begin{equation*}
\lambda \sum_{ j = 0 }^m p_j \binom{ m }{ j } u^j (1 - u)^{ m - j } - u,
\end{equation*}
for some $m \in \bbN$, $\lambda > 0$, and coefficients $p_j \in [ 0, 1 ]$ by setting $\lambda$ as the branching rate into $m$ particles, and suitably adapting the voting scheme.
Our contribution is to prove an analogous result for \eqref{allen_cahn_general}, subject to the same initial and boundary conditions as \eqref{allen_cahn_equation}, when $f$ is merely continuous and polynomially bounded.

Consider a branching Brownian motion reflected off the boundary $\partial \Omega$, with branching rate $\lambda$, started from a single particle at $x \in \Omega$ at time $0$.
At a branching event, the number of offspring is given by $\eta( f, L )$, where $L \sim \text{Geo}( 1 / 4 )$.
At a terminal time $t > 0$, a leaf particle at $z \in \Omega$ samples allele $a$ (resp.\ $A$) with probability $1 - u_0( z )$ (resp.\ $u_0( z )$), and each leaf carries out this choice independently.
Alleles are propagated rootwards along the tree: a branch with $\eta( f, k )$ offspring, $j$ of whom carry allele $A$, carries allele $A$ if $f_k( j / \eta( f, k ) ) \geq 1/2$, and carries allele $a$ otherwise.
Let $F(x, t)$ be the allele carried by the root particle at position $x \in \Omega$ when the branching process is run until time $t > 0$.

\begin{thm}
Suppose $f$ is continuous and polynomially bounded as in \eqref{polynomially_bounded}.
Viewed as a function of $x \in \Omega$, the probability $\bbP( F( x, t ) = A )$ solves \eqref{allen_cahn_general} subject to the initial and boundary conditions in \eqref{allen_cahn_equation}.
\end{thm}
\begin{proof}
The proof is an adaptation of that of \cite[Proposition 2.4]{etheridgeetal:2022} to our more general setting.
To verify that $q(x, t) := \bbP( F( x, t ) = A )$ solves \eqref{allen_cahn_general} on the interior of $\Omega$, let $S$ denote the first branching time of the initial particle and $( W_t )_{ t \geq 0 }$ be a Brownian motion.
For small $h > 0$,
\begin{align*}
&q( x, t + h ) \\
&= \bbP( F( x, t + h ) = A | S > h ) \bbP( S > h ) + \bbP( F( x, t + h ) = A | S \leq h ) \bbP( S \leq h ) \\
&= \bbE_x[ q( W_h, t ) | S > h ] \bbP( S > h ) + \bbP( S \leq h ) \sum_{ k = 1 }^{ \infty } \Big( \frac{ 3 }{ 4 } \Big)^{ k - 1 } \frac{ 1 }{ 4 } \sum_{ j = 0 }^{ \eta( f, k ) } \mathds{ 1 } \Bigg\{ f_k\Big( \frac{ j }{ \eta( f, k ) } \Big) \geq \frac{ 1 }{ 2 } \Bigg\} \binom{ \eta( f, k ) }{ j } \\
&\phantom{= \bbE_x[ q( W_h, t ) | S > h ] \bbP( S > h ) + \bbP( S \leq h )} \times \bbE_x[ q( W_S, t + h - S )^j ( 1 - q( W_S, t + h - S  ) )^{ \eta( f, k ) - j } | S \leq h ]  \\
&= \bbE_x[ q( W_h, t ) | S > h ] e^{ - \lambda h } + \bbE_x[ f( q( W_S, t + h - S ) ) | S \leq h ] ( 1 - e^{ - \lambda h } ),
\end{align*}
where the subscript in $\bbE_x$ denotes the starting point of $( W_t )_{ t \geq 0 }$ and the last equality follows via \eqref{pls_expansion}.
Since $f$ is continuous, regularity of the heat semigroup and the tower law (viewing $q(x, t)$ as the expectation of an indicator function) yield
\begin{equation*}
\bbE_x[ f( q( W_S, t + h - S ) ) | S \leq h ] =  f( \bbE_x[ q( W_S, t + h - S ) | S \leq h ] ) + O( h ) = f( q( x, t ) ) + O( h ).
\end{equation*}
Hence,
\begin{align*}
\partial_t \bbP( F( x, t ) = A ) &= \lim_{ h \to 0 } \frac{ \bbE_x[ \bbP( F( W_h, t ) = A ) | S > h ] - \bbP( F( x, t ) = A ) }{ h } e^{ - \lambda h } \\
&\phantom{=} + \lim_{ h \to 0 } \frac{  f( \bbP( F( x, t ) = A ) ) - \bbP( F( x, t ) = A ) }{ h } ( 1 - e^{ - \lambda h } ) \\
&= \Delta \bbP( F( x, t ) = A ) + \lambda [ f( \bbP( F( x, t ) = A ) ) - \bbP( F( x, t ) = A ) ].
\end{align*}
The boundary condition is inherited from reflecting Brownian motion \cite{bass/hsu:1991} since branching events occur at a finite rate and hence will not take place on the boundary.
\end{proof}

\section{Discussion}\label{discussion}

In \cite{corderoetal:2022}, the authors prove analogues of Theorems \ref{wf_theorem}--\ref{duality_theorem} for the case
\begin{equation}\label{cordero_drift}
f( x ) - x := \sum_{ \ell = 2 }^m \beta_{ \ell } \sum_{ i = 0 }^{ \ell } \binom{ \ell }{ i } x^i ( 1 - x )^{ \ell - i } \Big( p_{ i, \ell } - \frac{ i }{ \ell } \Big),
\end{equation}
for a fixed $m \in \bbN$, positive coefficients $\{ \beta_{ \ell } \}_{ \ell = 2 }^m$, and a sequence of $[0, 1]$-valued coefficients $\{ \{ p_{ i, \ell } \}_{ i = 0 }^{ \ell } \}_{ \ell = 2 }^m$.
Earlier work by Gonz\'alez Casanova and Span\`o also covered the case
\begin{equation*}
f( x ) := \sum_{ j = 0 }^{ \infty } \pi_j x^j,
\end{equation*}
where $\{ \pi_j \}_{ j = 0 }^{ \infty }$ is a probability mass function \cite{casanova/spano:2018}.
Our results cover all Lipschitz continuous, polynomially bounded functions $f : [ 0, 1 ] \mapsto [ 0, 1 ]$, which includes both of these classes as special cases.
Furthermore, \cite[Section 2.10]{corderoetal:2022} mentions that their approach should extend to the $m = \infty$ case.
Our Bernoulli factory approach demonstrates that this is true, and also that there is no further difficulty in handling our more general class of drift functions.

The works of \cite{casanova/spano:2018} and \cite{corderoetal:2022} are more general than our results in two ways: they tackle sequences of drift functions $f_N \to f$ as $N \to \infty$, and they incorporate jumps in the limiting forward-in-time Wright--Fisher diffusion, along with multiple mergers into the reverse-time ancestral process.
Both of these generalisations could be incorporated into our model at the cost of increased technicality.
We have chosen to omit them to focus on the class of drift functions, which is our main interest.

The link our work establishes between individual-based models and diffusive scaling limits provides rigorous justification for a range of diffusive approximations which have been obtained for non-neutral finite-population models \cite{taylor/nowak:2006, lessard/ladret:2007}.
In addition, we provide an associated genealogical description and a formal duality relation between the two processes.
Genealogical processes of Wright--Fisher diffusions with frequency-dependent selection are also considered in \cite{coop/griffiths:2004}, though their approach is to condition on the allele frequency trajectory and hence avoid the need for branching events in the ancestral process.
Their setting is formulated for generic drift functions of the form $\beta(x) x (1 - x)$, though in practice they focus on Bernstein polynomials of degree no more than two.

Approaches based on moment duality, with duality function $h(x, n) = x^n$, have been used to obtain series expansions of Wright--Fisher diffusion transition functions \cite{barbouretal:2000}, and \cite{corderoetal:2022} uses the Bernstein duality in \eqref{bernstein_duality} and the Bernstein coefficient process $( \mathbf{ V }_t )_{ t \geq 0 }$ to study fixation for drifts of the form \eqref{cordero_drift}.
It is unclear whether similar results can be usefully obtained in our setting in practice because the sequences $\{ f_k \}_{ k \geq 1 }$ and $\{ \eta( f, k ) \}_{ k \geq 1 }$ are difficult to compute.
Hence, so are the Bernstein coefficients.

The construction of the solution to the Allen--Cahn equation \eqref{allen_cahn_general} from a branching Brownian motion is also an example of duality between these two processes.
This result could also be generalised by replacing the Laplacian with a more general second order differential operator, provided that it has suitable regularity and generates a diffusion with tractable reflecting behaviour at boundaries.

The key advantage of the Keane--O'Brien factory is that it covers the whole class of continuous, polynomially bounded functions $f$.
Other factories could be used to obtain different ancestral processes, Wright--Fisher diffusions, and constructions of the solution to the Allen--Cahn equation.
However, it is essential that the number of coins needed by the factory is independent of the realisations of those coins, or at least that there is a almost surely finite upper bound on the number of coins regardless of realisations.
Otherwise, the trick of keeping track of all resulting branches to fill in alleles (or votes) later cannot work, because the number of branches cannot be determined until earlier alleles have been resolved.
As far as we are aware, the only other somewhat general Bernoulli factory with this independence property is that of \cite{mendo:2019}, which is identical to the branching mechanism used in  \cite{casanova/spano:2018} and applies to exactly the same functions.
Seminal Bernoulli factories, such as that of Nacu and Peres \cite{nacu/peres:2005} based on Bernstein polynomial approximation of $f$, inherently link the decision to keep flipping coins to the outcomes of earlier coins.

Finally, there are multivariate analogues of Bernoulli factories, in which independent, $m$-sided dice with probability mass function $( p_1, \ldots, p_m )$ are used to construct a $v$-sided die with mass function $f( p_1, \ldots, p_m )$. 
To date, attention has focused on domains which exclude boundaries, and where the coordinates of $f$ are rational functions \cite{morina:2021, morinaetal:2022}, or where $v = 1$ \cite{leme/schneider:2023}.
We believe a suitable extension of such \emph{dice enterprises} to cases including boundaries could be used to obtain analogues of the convergence and duality results presented here in the case with more than two non-neutral alleles.

\section*{Funding and data sharing statements}

JK was supported by EPSRC research grant EP/V049208/1.
K{\L} was supported by the Royal Society through the Royal Society University Research Fellowship.
Data sharing is not applicable to this article as no new data were generated or analysed.

\bibliographystyle{alpha}
\bibliography{bibliography}  

\begin{thebibliography}{M{\L}NW22}

\bibitem[BEG00]{barbouretal:2000}
A.~D. Barbour, S.~N. Ethier, and R.~C. Griffiths.
\newblock A transition function expansion for a diffusion model with selection.
\newblock {\em Ann. Appl. Probab.}, 10:123--162, 2000.

\bibitem[BH91]{bass/hsu:1991}
R.~F. Bass and P.~Hsu.
\newblock Some potential theory for reflecting brownian motion in {Holder} and
  {Lipschitz} domains.
\newblock {\em Ann. Probab.}, 19(2):486--508, 1991.

\bibitem[CG04]{coop/griffiths:2004}
G.~Coop and R.~C. Griffiths.
\newblock Ancestral inference on gene trees under selection.
\newblock {\em Theor. Popul. Biol.}, 66:219--232, 2004.

\bibitem[CHS22]{corderoetal:2022}
F.~Cordero, S.~Hummel, and E.~Schertzer.
\newblock General selection models: {Bernstein} duality and minimal ancestral
  structures.
\newblock {\em Ann. Appl. Probab.}, 32(3):1499--1556, 2022.

\bibitem[Dur08]{Durrett08}
R.~Durrett.
\newblock {\em Probability models for {DNA} sequence evolution}.
\newblock Springer, 2008.

\bibitem[EGL22]{etheridgeetal:2022}
A.~M. Etheridge, M.~D. Gooding, and I.~Letter.
\newblock On the effects of a wide opening in the domain of the (stochastic)
  {Allen}--{Cahn} equation and the motion of hybrid zones.
\newblock {\em Electron. J. Probab.}, 27:Article 1, 2022.

\bibitem[EGT10]{Etheridge10}
A.~M. Etheridge, R.~C. Griffiths, and J.~E. Taylor.
\newblock A coalescent dual process in a moran model with genic selection, and
  the {Lambda} coalescent limit.
\newblock {\em Theor. Popln Biol.}, 78(2):77--92, 2010.

\bibitem[EK86]{ethier/kurtz:1986}
S.~N. Ethier and T.~G. Kurtz.
\newblock {\em Markov Processes: Characterization and Convergence}.
\newblock John Wiley \& Sons, Inc., 1986.

\bibitem[Goo18]{gooding:2018}
M.~Gooding.
\newblock {\em Long term behaviour of spatial population models with
  heterozygous or asymmetric homozygous selection}.
\newblock PhD thesis, Oxford University, 2018.

\bibitem[GS18]{casanova/spano:2018}
A.~{Gonzalez Casanova} and D.~Span\`o.
\newblock Duality and fixation in {$\Xi$-Wright--Fisher} processes in
  frequency-dependent selection.
\newblock {\em Ann. Appl. Probab.}, 28(1):250--284, 2018.

\bibitem[HNP11]{holtzetal:2011}
O.~Holtz, F.~Nazarov, and Y.~Peres.
\newblock New coins from old, smoothly.
\newblock {\em Constr. Approx.}, 33:331--363, 2011.

\bibitem[Kal02]{kallenberg:2002}
O.~Kallenberg.
\newblock {\em Foundations of Modern Probability}.
\newblock Springer New York, NY, 2nd edition, 2002.

\bibitem[KN97]{Krone97}
S.~M. Krone and C.~Neuhauser.
\newblock Ancestral processes with selection.
\newblock {\em Theor. Popln Biol.}, 51(3):210--237, 1997.

\bibitem[KO94]{keane/obrien:1994}
M.~S. Keane and G.~L. {O'Brien}.
\newblock A {Bernoulli} factory.
\newblock {\em AMC Trans. Modeling and Computer Simulation}, 4:213--219, 1994.

\bibitem[{\L}KPR11]{latuszynskietal:2011}
K.~{\L}atuszy\'nski, I.~Kosmidis, O.~Papaspiliopoulos, and G.~O. Roberts.
\newblock Simulating events of unknown probabilities via reverse-time
  martingales.
\newblock {\em Random Struct. Algorithms}, 38(4):441--452, 2011.

\bibitem[LL07]{lessard/ladret:2007}
S.~Lessard and V.~Ladret.
\newblock The probabilty of fixation of a single mutant in an exchangeable
  selection model.
\newblock {\em J. Math. Biol.}, 54:721--744, 2007.

\bibitem[Men19]{mendo:2019}
L.~Mendo.
\newblock An asymptotically optimal {Bernoulli} factory for certain functions
  that can be expressed as power series.
\newblock {\em Stoch. Proc. Appl.}, 129:4366--4384, 2019.

\bibitem[M{\L}NW22]{morinaetal:2022}
G.~Morina, K.~{\L}atuszy\'nski, P.~Nayar, and A.~Wendland.
\newblock From the {Bernoulli} factory to a dice enterprise via perfect
  sampling of {Markov} chains.
\newblock {\em Ann. Appl. Probab.}, 32(1):327--359, 2022.

\bibitem[Mor21]{morina:2021}
G.~Morina.
\newblock {\em Extending the {Bernoulli} factory to a dice enterprise}.
\newblock PhD thesis, University of Warwick, 2021.

\bibitem[NP05]{nacu/peres:2005}
{\c{S}}.~Nacu and Y.~Peres.
\newblock Fast simulation of new coins from old.
\newblock {\em Ann. Appl. Probab.}, 15:93--115, 2005.

\bibitem[PLS23]{leme/schneider:2023}
R.~Paes~Leme and J.~Schneider.
\newblock Multiparameter {Bernoulli} factories.
\newblock {\em Ann. Appl. Probab.}, 33(5):3987--4007, 2023.

\bibitem[TN06]{taylor/nowak:2006}
C.~Taylor and M.~A. Nowak.
\newblock Evolutionary game dynamics with non-uniform interaction rates.
\newblock {\em Theor. Popul. Biol.}, 69:243--252, 2006.

\bibitem[{von}51]{vonNeumann:1951}
J.~{von Neumann}.
\newblock Various techniques used in connection with random digits.
\newblock {\em Appl. Math. Ser.}, 12:36--38, 1951.
\newblock Reprinted in \emph{von Neumann's Collected Works 5}, Oxford
  University Press, 1963, pp.\ 768--770.

\bibitem[XYZ19]{xietal:2019}
F.~Xi, G.~Yin, and C.~Zhu.
\newblock Regime-switching jump diffusions with non-{Lipschitz} coefficients
  and countably many switching states: existence and uniqueness, {Feller}, and
  strong {Feller} properties.
\newblock In G.~Yin and Q.~Zhang, editors, {\em Modelling, Stochastic Control,
  Optimization, and Applications}. Springer, Cham., 2019.

\end{thebibliography}

\end{document}